\documentclass[12pt,reqno]{amsart}

\usepackage[T1]{fontenc}
\usepackage[latin1]{inputenc}
\usepackage[english]{babel}

\usepackage{tikz-cd}
\usepackage[all]{xy}
\usepackage{hyperref}
\usepackage{amsmath,amssymb}
\usepackage{amsthm}
\usepackage{thmtools}

\newtheorem{theorem}{Theorem}[section]
\newtheorem{lemma}[theorem]{Lemma}
\newtheorem{coro}[theorem]{Corollary}
\newtheorem{prop}[theorem]{Proposition}
\theoremstyle{definition}

\newtheorem{example}[theorem]{Example}
\newtheorem{rem}[theorem]{Remark}

\newcommand{\cC}{\ensuremath{\mathcal{C}}}
\newcommand{\C}{\mathcal{C}}

\newcommand{\E}{\mathcal{E}}

\newcommand{\farf}{\mathrel{\text{\scalebox{1}{$\Join$}}}}
\newcommand{\RG}{\mathbf{RG}}

\newcommand{\Pt}{\mathbf{Pt}}

\newcommand{\Ext}{\mathbf{Ext}}
\newcommand{\CExt}{\mathbf{CExt}}
\newcommand{\Lie}{\mathbf{Lie}}
\newcommand{\XMod}{\mathbf{XMod}}
\newcommand{\PXMod}{\mathbf{PXMod}}
\newcommand{\PXB}{\mathbf{PXMod}_{B} (\mathcal C)}
\newcommand{\XB}{\mathbf{XMod}_{B} (\mathcal C)}
\newcommand{\RGB}{\mathbf{RG}_{B}(\mathcal C)}
\newcommand{\GPD}{\mathbf{Grpd}_{B}(\mathcal C)}
\newcommand{\SH}{\ensuremath{\mathrm{(SH)}}}
\newcommand{\UA}{\ensuremath{\mathrm{(UA)}}}
\newcommand{\Ker}{\ensuremath{\mathrm{Ker}}}
\newcommand{\PX}{\ensuremath{{_\mathsf{PX}}}}

\newdir{|>}{{}*!/8pt/@{|}*!/3.5pt/:(1,-.2)@^{>}*!/3.5pt/:(1,+.2)@_{>}}
\newdir{ >}{{}*!/-10pt/@{>}}
\newdir{ |>}{!/10pt/{}*!/4.5pt/@{|}*:(1,-.2)@^{>}*:(1,+.2)@_{>}}

\title{Galois theory and the categorical Peiffer commutator}
\author{A. S. Cigoli, A. Duvieusart, M. Gran and S. Mantovani}
\thanks{The second author is a Research Fellow of the Fonds de la Recherche Scientifique-FNRS}
\date{\today}

\begin{document}
\begin{abstract}
We show that the Peiffer commutator previously defined by Cigoli, Mantovani and Metere can be used to characterize central extensions of precrossed modules with respect to the subcategory of crossed modules in any semi-abelian category satisfying an additional property. We prove that this commutator also characterizes double central extensions, obtaining then some Hopf formulas for the second and third homology objects of internal precrossed modules.
\end{abstract}

\maketitle
\section{Introduction and preliminaries} \label{sec1}
Let $\C$ be a semi-abelian category \cite{JMT} satisfying the ``\emph{Smith is Huq}'' condition, denoted by $\SH$ \cite{BG, NT} in the following. This condition means that two notions of centrality coincide: the first one is the notion of centrality for equivalence relations (in particular, of congruences in varieties of universal algebras) \cite{Smith, Ped}  and the second one is the centrality (often referred to as commutativity) of the corresponding normal subobjects (in particular, of normal subalgebras) \cite{Huq}. Thanks to this coincidence, in the present article we mainly work with the latter notion, that we are now going to recall. In $\mathcal C$, two subobjects $m \colon M \rightarrow A$ and $n \colon N \rightarrow A$ of the same object $A$ \emph{commute in the sense of Huq} if there is an arrow $c \colon M \times N \rightarrow A$ making the diagram
$$
\xymatrix{ M \ar[r]^-{(1,0) } \ar[dr]_m & M \times N \ar@{-->}[d]^{c} & N \ar[l]_-{(0, 1)} \ar[dl]^n \\
&A  & }
$$
commute. When this is the case, such an arrow $c$ is unique, and it is called the \emph{cooperator} of $m$ and $n$ \cite{Bourn2002}. With a slight abuse of notation we write $[M, N]_{Huq}^A = 0$ in this case, without explicitly mentioning the morphisms $m$ and $n$, or simply $[M, N]_{Huq} = 0$.
Given any two normal subobjects $m \colon M \rightarrow A$ and $n \colon N \rightarrow A$ as above, there is in $\mathcal C$ \cite{Bourn2004} a \emph{smallest normal subobject} $L$ of $A$ such that in the quotient $A/L$ the regular images $q(M)$ and $q(N)$ along $q \colon A \rightarrow A/L$ commute: $$[q(M), q(N)]_{Huq}^{A/L}= 0.$$
Such a subobject is usually denoted by $[M, N]_{Huq}^A$; moreover, $[M, N]_{Huq}^A$ is the trivial subobject $0 \rightarrow A$ of $A$ if and only if $m$ and $n$ commute in the sense of Huq, so that the notations are consistent.

Since the condition $\SH$ holds in $\mathcal C$, a reflexive graph 
\begin{equation}\label{reflexive}
\xymatrix{
X_1\ar@<5pt>[r]^-{d} \ar@<-5pt>[r]_-{c} & B \ar[l]|-{e}
}\end{equation}
(with $d \cdot  e = 1_{{B}}= c \cdot e$) is an internal groupoid if and only if the kernels $\ker(d) \colon K[d] \rightarrow X_1$ and $\ker(c) \colon K[c] \rightarrow X_1$ of the ``domain'' $d$ and of the ``codomain'' $c$ have trivial Huq commutator: $[K[d], K[c]]_{Huq}^{X_1}= 0$ (see \cite{Ped, NT}).
One writes $\RGB$ for the category of reflexive graphs in $\mathcal C$ over a fixed ``object of objects'' $B$, with morphisms  those $f_1 \colon X_1 \rightarrow Y_1$ in $\mathcal C$ such that in the diagram
\begin{equation}\label{morphism}
\xymatrix{
X_1 \ar@<1ex>[dr]^{c} \ar@<-1ex>[dr]_{d} \ar[rr]^{f_1} & & Y_1 \ar@<1ex>[dl]^{c'} \ar@<-1ex>[dl]_{d'} \\
& B \ar[ul]|e \ar[ur]|{e'}}
\end{equation}
the obvious triangles commute. Since $\mathcal C$ is semi-abelian, the category $\RGB$ is also \emph{exact} \cite{Barr}, with regular epimorphisms those morphisms such that $f_1$ in \eqref{morphism} is a regular epimorphism in $\mathcal C$, and \emph{protomodular} \cite{Bourn1991}. This category $\RGB$ is not pointed, but \emph{quasi-pointed} \cite{Bourn2001}, in the sense that it has an initial object $(B,1_B, 1_B, 1_B)$, a terminal object $(B \times B, p_1, p_2, ( 1_B, 1_B ) )$ and, moreover, the canonical arrow from the initial to the terminal object is a monomorphism. 

The category $\RGB$ is known to be equivalent to the category $\PXB$ of (internal) \emph{precrossed modules} \cite{J03} over a fixed object $B$, also studied in \cite{MM, CMM17}. 

The normalization functor $N \colon \RGB \rightarrow \PXB$ giving this category equivalence associates, with any reflexive graph \eqref{reflexive}, the precrossed module
$(\partial\colon X  \to B,\xi)$, where $\partial = c \cdot \ker (d)$, $X = K[d]$, and $\xi \colon B \flat X \rightarrow X$ is the \emph{internal action} (in the sense of \cite{BJK}, see the next section for details) given by the conjugation of $B$ on $X$, computed in $X_1$. Note that, by definition, the action $\xi$ of a precrossed module $(\partial\colon X \to B,\xi)$ makes the diagram 
$$
\xymatrix{B \flat X \ar[r]^{1 \flat \partial} \ar[d]_\xi& \ar[d] B \flat B \ar[d]^{\chi} \\
X \ar[r]_\partial & B
}
$$
commute, with $\chi$ the conjugation action of $B$ on itself. For instance, in the case of groups, the commutativity of this diagram expresses, internally, the precrossed module condition $\partial ({\, }^b x ) = bxb^{-1}$.

The normalization functor $N \colon \RGB \rightarrow \PXB$ takes a morphism \eqref{morphism} to the morphism
\begin{equation}\label{normalized}
\xymatrix{ X  \ar[rr]^-{f} \ar[dr]_{\partial = c \cdot \ker(d)}&   & Y  \ar[dl]^{\partial' = c' \cdot \ker (d') } \\
&B & }
\end{equation}
where $f$ is the restriction of $f_1$ to the kernels $X$ and $Y$ of $d$ and of $d'$, respectively, whence $\partial' \cdot f = \partial$. From the point of view of the actions, $f$ is \emph{equivariant} with respect to the $B$-actions, in the sense that the following diagram commutes:
$$
\xymatrix{
B \flat X \ar[r]^{1 \flat f} \ar[d]_{\xi}& \ar[d] B \flat Y \ar[d]^{\xi'} \\
X \ar[r]_f & Y,
}
$$
so that $f \colon (\partial\colon X \to B,\xi) \rightarrow (\partial' \colon Y \to B,{\xi'})$ is a precrossed module morphism.

By the definition of internal crossed module given in \cite{J03} the category equivalence $\RGB \cong \PXB$ restricts to an equivalence between the category $\GPD$ of \emph{internal groupoids} in $\mathcal C$ over $B$ 
and the category $\XB$ of \emph{internal crossed modules} over $B$.
The condition $\SH$ in $\mathcal C$ means precisely that a precrossed module $(\partial\colon X \to B,\xi)$ is a \emph{crossed module} if and only if the following diagram  
$$
\xymatrix{X \ar[d]_{\chi} \flat X \ar[r]^{\partial \flat 1} & B \flat X \ar[d]^{\xi} \\
X \ar@{=}[r]_{1_X}& X
}
$$
commutes (see \cite{MM,NT}). 

The category $\GPD$ is a full reflective subcategory of the category $\RGB$:
\begin{equation}\label{adjunction-RG}
\xymatrix@=30pt{
{ \RGB  } \ar@<1ex>[r]_-{^{\perp}}^-{F} & {\, \GPD.\, }
\ar@<1ex>[l]^-U  }
 \end{equation}
Under our assumptions on $\mathcal C$, the $(X_1,d,c,e)$-component of the unit of this adjunction is given by the quotient 
\begin{equation}\label{reflection-RG}
 \xymatrix@=25pt{
	X_1 \ar@<-1,ex>[d]_{d}\ar@<+1,ex>[d]^{c} \ar@{->>} [rr]^-{\eta_{X_1}}
	& & {\frac{X_1}{[K[d],K[c]]_{Huq}}\, } \ar@<-1,ex>[d]_{\overline{d}}\ar@<+1,ex>[d]^{\overline{c}}   \\
	B  \ar@{=}[rr]_{} \ar[u]|-{e}& &{B\, } \ar[u]|-{\overline{e}} 
}
\end{equation}
where ${[K[d],K[c]]_{Huq}}$ is the Huq commutator in $X_1$ of the kernels of $d$ and $c$. Thanks to the category equivalences recalled above, one knows that $\XB$ is a reflective subcategory of $\PXB$:
\begin{equation}\label{adjunction-PX}
\xymatrix@=30pt{
{ \PXB  } \ar@<1ex>[r]_-{^{\perp}}^-{G} & {\, \XB.\, }
\ar@<1ex>[l]^-V  }
 \end{equation}

A categorical notion of Peiffer commutator was introduced in \cite{CMM17} (see the next section), and the reflection of the precrossed $B$-module $(\partial\colon X \to B,\xi)$ associated with the reflexive graph $(X_1,d,c,e)$ was shown to be the quotient $\eta_X \colon X \rightarrow \frac{X}{\langle X,X \rangle} $ of $X$ by the \emph{Peiffer commutator} $\langle X, X \rangle$ on $X$
\begin{equation}\label{Reflection-PX}
\xymatrix{ X  \ar@{->>}[rr]^-{\eta_X} \ar[dr]_{\partial}&   & \frac{X}{\langle X,X \rangle}  \ar[dl]^{\overline{\partial}} \\
&B & }
\end{equation}
where the $B$-action $\overline{\xi}$ on $\frac{X}{\langle X,X \rangle} $ is the one induced by the $B$-action $\xi$ on $X$.

The correspondence between the Peiffer commutator $\langle X,X \rangle$ on $X$ in \eqref{Reflection-PX} and the Huq commutator $[K[d],K[c]]_{Huq}$ in the reflection \eqref{reflection-RG} raises the question of determining whether this is a special case of a more general fact relating centrality conditions coming from categorical Galois theory \cite{JK} to this Peiffer commutator (in a context where they are both defined and can then be compared). The interest for this question also comes from a recent result in Galois theory that we now briefly explain. 

A characterization of the \emph{extensions} in $\RGB$ that are \emph{central} with respect to the adjunction \eqref{adjunction-RG} was established in \cite{DG}, in the general context of exact Mal'tsev categories, i.e.~in exact categories where any reflexive relation is an equivalence relation \cite{CLP}. 
Recall that a \emph{Birkhoff subcategory} is simply a full regular epi-reflective subcategory $\mathcal X$ of a category  $\mathcal A$ 
\begin{equation}\label{adjunction-JK}
\xymatrix@=30pt{
{ \mathcal A } \ar@<1ex>[r]_-{^{\perp}}^-{G} & {\, {\mathcal X}\, }
\ar@<1ex>[l]^-V  }
 \end{equation}
that is stable in $\mathcal A$ under regular quotients. As explained in \cite{JK}, when $\mathcal A$ is an exact Mal'tsev category, a Birkhoff subcategory $\mathcal X$ of $\mathcal A$ always induces an \emph{admissible} Galois structure, for which there is a classification theorem of the extensions that are $\mathcal X$-central, in a sense that we are now going to recall. 
An extension $f \colon  \xymatrix{X \ar@{->>}[r] & Y }$ in $\mathcal A$ is called an $\mathcal X$-\emph{trivial extension} when the naturality square 
$$\xymatrix{X \ar@{->>}[r]^{f} \ar@{->>}[d]_{\eta_X}& Y \ar@{->>}[d]^{\eta_Y} \\  VG(X) \ar@{->>}[r]_{VG(f) } & VG(Y)} $$
induced by the unit $\eta$ of the adjunction \eqref{adjunction-JK} is a pullback. The notion of ${\mathcal X}$-central extension is then defined as an extension in $\mathcal A$ that is \emph{locally $\mathcal X$-trivial}, in the sense that it is ${\mathcal X}$-trivial up to the pullback in $\mathcal A$ along a regular epimorphism (= an effective descent morphism, in this context \cite{JT}). 
In other words, a regular epimorphism $f \colon  \xymatrix{X \ar@{->>}[r] & Y }$ in $\mathcal A$ is called an $\mathcal X$-\emph{central extension}
 if there is a regular epimorphism $p \colon  \xymatrix{Z \ar@{->>}[r] & Y }$ in $\mathcal A$ such that the projection $p_1$ in the pullback 
$$\xymatrix{ Z \times_Y X \ar@{->>}[r]^-{p_2} \ar@{->>}[d]_{p_1}& X \ar@{->>}[d]^{f} \\  Z \ar@{->>}[r]_{p} & Y} $$
is an $\mathcal X$-trivial extension. In particular, $f$ is called an $\mathcal X$-\emph{normal extension} if in the above diagram we can take $p=f$. We recall from \cite{JK} that, when $\mathcal A$ is Mal'tsev, every central extension is normal.

When $\mathcal C$ is exact Mal'tsev (as it follows from our assumptions), the category $\mathcal A = \RGB$ is again exact Mal'tsev, so that it is natural to investigate which are the \emph{extensions} \begin{equation}
\xymatrix@=25pt{
X_1 \ar@<1ex>[dr]^{c} \ar@<-1ex>[dr]_{d} \ar@{->>}[rr]^{f_1} & & Y_1 \ar@<1ex>[dl]^{c'} \ar@<-1ex>[dl]_{d'} \\
& B \ar[ul]|e \ar[ur]|{e'}}
\end{equation}
in $\RGB$ that are $\GPD$-\emph{central}, namely central with respect to the Birkhoff reflection \eqref{adjunction-RG}. As shown in \cite{DG} (by extending a result in \cite{EG06}), it turns out that this is the case if and only if the following Smith \emph{centrality condition} holds:
\begin{equation} \label{centralRG}
[Eq[f_1],Eq[c]\vee Eq[d]]_{Smith}=\Delta_{X_1}.
\end{equation}
Here $Eq[c]\vee Eq[d]$ is the supremum of the equivalence relations $Eq[c]$ and $Eq[d]$ that are the kernel pairs of the morphisms $c$ and $d$, respectively, while $\Delta_{X_1}$ is the discrete relation on $X_1$.
The results in \cite{BG3, Ped} imply that this condition is equivalent to the following ones:
\begin{equation}\label{CC}  
[Eq[f_1],Eq[c]]_{Smith} = \Delta_{X_1}, \quad  \mathrm{and} \quad [Eq[f_1], Eq[d]]_{Smith}={\Delta_{X_1}.}
\end{equation}
When we look at conditions \eqref{CC} in terms of the \emph{Huq centrality}, thanks to the $\SH$ condition, we can express them as follows:
\begin{equation} \label{char.centr}
[K [f_1], K[c] ]_{Huq} = 0 \quad  \mathrm{and} \quad  [K [f_1], K[d] ]_{Huq}=0.
\end{equation}
In the next section, after recalling some useful definitions, we shall see that, under suitable assumptions on the base category $\mathcal C$, these conditions are equivalent to asking that the Peiffer commutator $\langle K[f] , X \rangle$ is trivial, where $f$ is the extension in $\PXB$ corresponding to $f_1$ via the normalization functor.

In the third section we shall use this characterization and a result in \cite{EG07} to get a five term exact sequence in homology (Proposition \ref{SS}), where the homology objects in $\PXB$ are expressed in terms of generalized Hopf formulas. When $\mathcal C$ is the category of Lie algebras, one obtains an exact sequence in the category of Lie algebra precrossed modules (see Remark \ref{Lie}). In the last section a characterization of ``double central extensions'' relative to the induced adjunctions between the categories of extensions and of central extensions in $\PXB$ will also be established (Theorem \ref{double-central}). From this, an explicit Hopf formula describing the Galois group of a weakly universal double central extension will be deduced (see formula (\ref{double-Hopf})).

\section{Main result}

The notions of internal precrossed and crossed module are based on \emph{internal actions} \cite{BJK}. For each object $B$ in a semi-abelian category \cC, one can consider the category $\Pt_B(\cC)$ of \emph{points} over $B$, whose objects are pairs $(p,s)$ of arrows in \cC\ with $ps=1_B$, and whose morphisms are triangles
\[
\xymatrix{
A \ar@<.5ex>[dr]^p \ar[rr]^f & & A' \ar@<.5ex>[dl]^{p'} \\
& B \ar@<.5ex>[ul]^s \ar@<.5ex>[ur]^{s'}
}
\]
where $fs=s'$ and $p'f=p$. The functor
\[
\Ker_B\colon \Pt_B(\cC)\to \cC,
\]
sending each point $(p,s)$ over $B$ to the kernel of $p$, and a map $f$ to its restriction to the kernels, has a left adjoint sending each object $X$ in \cC\ to the point
\[
\xymatrix{
B+X \ar@<.5ex>[r]^-{[1,0]} & B \ar@<.5ex>[l]^-{\iota_B}.
}
\]
The kernel of $[1,0]$ is usually denoted by $B\flat X$, and $B\flat(-)\colon\cC\to\cC$ is the underlying functor of the monad on \cC\ associated with the adjunction above. Internal $B$-actions are defined as the algebras for the monad $B\flat(-)$. In the semi-abelian context, the functor $\Ker_B$ is monadic, and there is then an equivalence
\[
\cC^{B\flat(-)}\simeq \Pt_B(\cC)
\]
between $B$-actions and points over $B$. In other words, \cC\ has \emph{semi-direct products} in the sense of \cite{BJ}. Explicitly, each point $(p,s)$ over $B$ determines a $B$-action $\xi$ given by the (unique) leftmost vertical arrow in the commutative diagram
\[
\xymatrix@C=6ex{
B\flat X \ar@{ |>->}[r]^-{\ker[1,0]} \ar[d]_{\xi} & B+X \ar@<.5ex>[r]^-{[1,0]} \ar[d]_{[s,\ker(p)]} & B \ar@<.5ex>[l]^-{\iota_B} \ar@{=}[d] \\
X \ar@{ |>->}[r]_-{\ker(p)} & A \ar@<.5ex>[r]^{p} & B. \ar@<.5ex>[l]^{s}
}
\]
If \cC\ is the category of groups, the group $B\flat X$ is generated as a subgroup of $B+X$ by the strings of the form $(b;x;b^{-1})$ with $b$ in $B$ and $x$ in $X$, and $\xi$ maps such generator to the element $s(b)xs(b)^{-1}$ of $X$, i.e.\ $\xi$ realizes internally the conjugation action of $B$ on $X$ inside $A$. Conversely, each internal action $\xi$ determines a point as in the right hand side of the bottom row of the diagram
\[
\xymatrix@C=6ex{
B\flat X \ar@{ |>->}[r]^-{\ker[1,0]} \ar[d]_{\xi} & B+X \ar@<.5ex>[r]^-{[1,0]} \ar[d]_{[i_B,j_X]} & B \ar@<.5ex>[l]^-{\iota_B} \ar@{=}[d] \\
X \ar@{ |>->}[r]_-{j_X} & X \rtimes_{\xi} B \ar@<.5ex>[r]^-{p_B} & B, \ar@<.5ex>[l]^-{i_B}
}
\]
where the left hand square is a pushout (notice that, by monadicity, $j_X$ is indeed the kernel of $p_B$). Again, in the category of groups, $X\rtimes_\xi B$ is the classical semi-direct product of groups.

Three special cases of internal actions deserve to be described:
\begin{itemize}
\item the \emph{trivial action} of $B$ on $X$, given by  the composite
\[
\tau \colon
\xymatrix@C=5ex{
B \flat X \ar@{ |>->}[r]^-{\ker[1,0]} & B+X \ar[r]^-{[0,1]} & X,
}
\]
and corresponding to the point
\[
\xymatrix{
B \times X \ar@<.5ex>[r]^-{p_1} & B; \ar@<.5ex>[l]^-{(1,0)}
}
\]
\item for a normal subobject $k\colon\xymatrix{K \ar@{ |>->}[r] & X}$, the \emph{conjugation action} of $X$ on $K$, given by the (unique) left vertical arrow in the commutative diagram
\[
\xymatrix{
X \flat K \ar[d]_{\chi^X_K} \ar@{ |>->}[r]^{\ker[1,0]} & X+K \ar[d]^{[1,k]} \\
K \ar@{ |>->}[r]_{k} & X,
}
\]
and corresponding to the point
\[
\xymatrix{
R \ar@<.5ex>[r]^-{p_1} & X, \ar@<.5ex>[l]^-{(1,1)}
}
\]
where $R$ is the equivalence relation on $X$ associated with $K$ (as a special case, we shall simply denote by $\chi\colon X\flat X \to X$ the conjugation action of $X$ on itself induced by the indiscrete relation);
\item for each action $\xi\colon B\flat X \to X$ and each morphism $f\colon A \to B$, the \emph{pullback action}, given by the composite
\[
f^*(\xi) \colon
\xymatrix{
A \flat X \ar[r]^-{f \flat 1_X} & B \flat X \ar[r]^-{\xi} & X,
}
\]
and corresponding to the upper point in the pullback diagram
\[
\xymatrix@C=6ex{
(X \rtimes_\xi B)\times_B A \ar@<.5ex>[r]^-{p_2} \ar[d]_{p_1} & A \ar[d]^{f} \ar@<.5ex>[l]^-{(i_Bf,1)} \\
X \rtimes_\xi B \ar@<.5ex>[r]^-{p_B} & B \ar@<.5ex>[l]^-{i_B}.
}
\]

\end{itemize}

The \emph{Peiffer product} of two precrossed $B$-modules $(\partial\colon X \to B,\xi)$ and $(\partial'\colon Y \to B, \xi')$ in \cC\ was introduced in \cite{CMM17} and can be defined as the object in the bottom right corner of the diagram
\begin{equation} \label{diag:farf}
\xymatrix@C=9ex{
    X +_\PX Y \ar[r]^-{[j_X,i_Y]_\PX} \ar[d]_{[i_X,j_Y]_\PX} & X \rtimes Y \ar[d] \\
    Y \rtimes X \ar[r] & X \farf Y,
}
\end{equation}
which has to be interpreted as the image of a pushout in $\PXMod_B(\cC)$ under the forgetful functor sending each precrossed module to the domain of its structure morphism ($X+_\PX Y$ denotes the domain of the coproduct of $X$ and $Y$ in $\PXB$ and both the semi-direct products above have a canonical precrossed $B$-module structure determined by those on $X$ and $Y$, as explained in \cite{CMM17}). We may denote by $\Sigma\colon X+_\PX Y \to X \farf Y$ the diagonal of the pushout \eqref{diag:farf}.

In \cite{CE}, Conduch\'e and Ellis defined the \emph{Peiffer commutator} $\langle X,Y \rangle$ of two precrossed $B$-submodules (of groups)
\begin{equation} \label{cospan}
\xymatrix{
X \ar@{ >->}[r]^m \ar[dr]_{\partial} & A \ar[d]^(.4){\partial_A} & Y \ar@{ >->}[l]_{n} \ar[dl]^{\partial'} \\
& B
}
\end{equation}
as the subgroup of $A$ generated by the elements of the form $xyx^{-1}(^{\partial(x)}y)^{-1}$ and $yxy^{-1}(^{\partial'(y)}x)^{-1}$. An internal version of this was defined in \cite{CMM17} for a general semi-abelian category, as the regular image, through the arrow $[m,n]_\PX\colon X +_\PX Y \to A$, of the kernel $N$ of the diagonal of the pushout \eqref{diag:farf}:
\begin{equation} \label{diag:peiff.comm}
\xymatrix{
    N \ar@{ |>->}[d] \ar@{->>}[r] & \langle X,Y \rangle \ar@{ >->}[d] \\
    X +_\PX Y \ar[r]_-{[m,n]_\PX} & A.
}
\end{equation}

\begin{rem} \label{rem:normal}
We recall from Remark 3.12 in \cite{CMM17} that, when $X$ and $Y$ act trivially on each other, the normal closure of their Peiffer commutator coincides with their Huq commutator. In particular, this is the case when both are normal precrossed submodules (which implies that $\partial$ and $\partial'$ are zero maps). 
\end{rem}

\begin{rem}
Notice that the Peiffer commutator of two precrossed submodules as in (\ref{cospan}) is not normal in general. However, it is the case when $A$ is the join of $X$ and $Y$ in $\PXMod_B(\cC)$ (see Remark 3.9 in \cite{CMM17}). In particular, this happens when considering $\langle X,K \rangle$ for some $K$ normal subobject of $X$ in $\PXMod_B(\cC)$. Moreover, we have the following lemma:
\end{rem}

\begin{lemma} \label{lemma:normal}
For a normal precrossed submodule
\[
\xymatrix{
K \ar[r]^{k} \ar[dr]_{0} & X \ar[d]^{\partial} \\
& B
}
\]
the inequality $\langle X,K \rangle\leq K$ holds.
\end{lemma}

\begin{proof}
First of all, let us notice that the trivial precrossed module map
\[
0\colon (0:K\to B,\xi_K)\to (\partial\colon X \to B, \xi_X) 
\]
exists, and so does $[1,0]_\PX\colon X+_{\PX} K\to X$. Moreover $X \rtimes_{0^*\xi_X} K\cong X\times K$ ($K$ acts trivially on $X$). Hence, specializing the pushout \eqref{diag:farf} to our context, by the commutativity of the external square in the diagram
\[
\xymatrix@C=9ex{
  X +_\PX K \ar[r]^-{[j_X,i_K]_\PX} \ar[d]_{[i_X,j_K]_\PX} & X \times K \ar[d] \ar@/^/[ddr]^{p_1} \\
  K \rtimes_{\partial^*\xi_K} X \ar[r] \ar@/_/[drr]_{p_X} & X \farf K \ar@{-->}[dr]^{\tau} \\
  & & X
}
\]
we get a unique arrow $\tau$ such that $\tau\cdot\Sigma=[1,0]_\PX$. Now, we can proceed as in Section 6 of \cite{MM}, and consider the diagram
\[
\xymatrix@C=8ex{
N \ar@{->>}[r] \ar@{ |>->}[d]_{\ker\Sigma} & \langle X,K \rangle \ar@{ |>->}[d] \ar@{ >-->}[r] & K \ar@{ |>->}[dl]^k \\
X+_\PX K \ar@{}[dr]|{(a)} \ar@{->>}[r]^{[1,k]_\PX} \ar@{->>}[d]^{\Sigma} \ar@/_5ex/@{->>}[dd]_{[1,0]_\PX} & X \ar@{->>}[d]^q \ar@/^5ex/@{->>}[dd]^p \\
X \farf K \ar@{->>}[r] \ar@{->>}[d]^{\tau} & \frac{X}{\langle X,K \rangle} \ar@{-->}[d] \\
X \ar@{->>}[r]_p & \frac{X}{K},
}
\]
where $p$ is the cokernel of $k$. It is easy to check that $p\cdot[1,k]_\PX=p\cdot[1,0]_\PX$ by precomposition with the canonical injections. The square $(a)$ is a pushout, since $q$ and $\Sigma$ are cokernels with a regular epimorphic comparison morphism between the corresponding kernels. By universal property we get that $p$ factors through $q$ and hence $\langle X,K \rangle\leq K$.
\end{proof}

\begin{prop} \cite[Proposition 3.11]{CMM17} \label{farf}
The Peiffer commutator $\langle X,Y \rangle$ of two precrossed $B$-submodules as in \eqref{cospan} is trivial if and only if there exists a (unique) morphism $\varphi$ making the diagram
\[
\xymatrix{
    X \ar[r]^-{l_{X}} \ar@{ >->}[dr]_{m} & X \farf Y \ar@{-->}[d]^{\varphi} & Y \ar[l]_-{l_Y} \ar@{ >->}[dl]^{n} \\
    & A.
}
\]
commute.
\end{prop}

\begin{prop} \cite[Proposition 3.13 and Corollary 3.14]{CMM17} \label{prop.peiffer}
The Peiffer commutator is preserved by regular images: if $q\colon A \to A'$ is a regular epimorphism in $\PXMod_B(\cC)$ and $X$ and $Y$ are precrossed $B$-submodules of $A$ as in \eqref{cospan}, then $q(\langle X,Y \rangle)=\langle q(X),q(Y) \rangle$.

The Peiffer commutator is monotone: if $X\leq X'$ and $Y \leq Y'$ are precrossed $B$-submodules of a given precrossed $B$-module $A$, then $\langle X,Y \rangle \leq \langle X',Y' \rangle$.
\end{prop}

Finally, we recall a condition, also introduced in \cite{CMM17}, that one may ask on a semi-abelian category \cC, and that turns out to be crucial in order to prove Theorem \ref{thm:main}:
\begin{itemize}
    \item [\UA] Given an extremal epimorphic cospan $\xymatrix{A \ar[r]^f & B & C \ar[l]_g}$ in \cC, then for any 4-tuple $(\xi_1,\xi_2,\xi_3,\xi_4)$ of actions on a fixed object $X$ making the diagram
\begin{equation} \label{diag:C}
\begin{aligned}
\xymatrix{
    A \flat X \ar[dr]_{\xi_1} \ar[r]^{f\flat 1} & B\flat X \ar@<-.5ex>[d]_(.4){\xi_3} \ar@<.5ex>[d]^(.4){\xi_4}
        & C \flat X \ar[dl]^{\xi_2} \ar[l]_{g\flat 1} \\
    & X
}
\end{aligned}
\end{equation}
commute, we have $\xi_3=\xi_4$.
\end{itemize}
As proved in \cite{CMM17}, this property holds in any \emph{action representable} semi-abelian category (see \cite{BJK}) and in any \emph{category of interest} in the sense of Orzech \cite{Orzech}, so the categories of groups, Lie and Leibniz algebras over a fixed field, rings, associative algebras, Poisson algebras over a commutative ring with unit are all examples of such. Note that the property \UA\ implies the property \SH\ recalled in Section \ref{sec1} (see \cite{Cig}). We are now ready to state the main result of this paper.

\begin{theorem} \label{thm:main}
Let \cC\ be a semi-abelian category satisfying \UA, and $B$ an object in \cC. An extension
\begin{equation} \label{diag:f}
\xymatrix{ X  \ar@{->>}[rr]^-{f} \ar[dr]_{\partial}&   & Y  \ar[dl]^{\partial'} \\
&B & }
\end{equation}
of precrossed $B$-modules in \cC\ is $\XMod_B(\C)$-central if and only if
\[
\langle K[f] , X \rangle = 0. 
\]
\end{theorem}
 
\begin{proof}
We first prove that if
\[\xymatrix{
(\partial_{X\times_{Y} Z}\colon X\times_{Y} Z \to B,\xi_{X\times_{Y} Z}) \ar@{->>}[r]^-{g'} \ar@{->>}[d]_{f'} & (\partial\colon X\to B,\xi )\ar@{->>}[d]^{f} \\
(\partial_{Z}\colon Z\to B,\xi_Z) \ar@{->>}[r]_{g} & (\partial'\colon Y\to B, \xi')
}\]
is a pullback and $g$ a regular epimorphism in the category $\PXB$, then $\langle K[f'],X\times_Y Z\rangle =0$ if and only if $\langle K[f],X\rangle =0$. We recall that such a pullback gives in particular a square
\[\xymatrix{
X\times_Y Z \ar@{->>}[r]^-{g'} \ar@{->>}[d]_{f'} & X\ar@{->>}[d]^{f} \\
Z\ar@{->>}[r]_{g} & Y
}\]
that is a pullback in $\C$, with $g$ a regular epimorphism in $\C$. This implies that $g'$ is also a regular epimorphism, and that $g'(\ker(f'))=\ker(f)$. Assuming first that $\langle K[f'],X\times_Y Z\rangle =0$, we then find that
\[\langle K[f],X\rangle =\langle g'(K[f']),g'(X\times_Y Z)\rangle =g'\left(\langle K[f'],X\times_Y Z\rangle\right)=0\]
because the Peiffer commutator is preserved under regular images by Proposition \ref{prop.peiffer}. Assuming now that $\langle K[f],X\rangle=0$, the same reasoning shows that $g'\left(\langle K[f'],X\times_Y Z\rangle\right)=0$. Moreover
\[f'\left(\langle K[f'],X\times_Y Z\rangle\right)=\langle f'(K[f']),Z\rangle =\langle 0,Z\rangle = 0.\]
Since $f'$ and $g'$ are jointly monic, this implies that $\langle K[f'],X\times_Y Z\rangle=0$.

Now $\langle K[f],X\rangle \leq \langle X,X\rangle$, so that any extension $f$ between crossed modules must satisfy $\langle K[f],X\rangle=0$. The previous argument then implies that the same is true for all trivial extensions with respect to \eqref{adjunction-PX}, since by definition a trivial extension is the pullback of an extension of crossed modules. This in turn implies that every central extension satisfies $\langle K[f],X\rangle=0$, since an extension is central if there exists a regular epimorphism $g$ such that the pullback of $f$ along $g$ is a trivial extension, and this proves the ``only if'' part.

Concerning the ``if'' part, let us first observe that, for any morphism \eqref{morphism}
in $\RG_B(\cC)$, the pullback
\begin{equation} \label{diag:K_1}
\xymatrix{
K_1 \ar@{ >->}[r]^{k_1} \ar@{->>}[d]_{p} & X_1 \ar@{->>}[d]^{f_1} \\
B \ar@{ >->}[r]_{e'} & Y_1
}
\end{equation}
determines a kernel (in the sense of quasi-pointed categories) of $f_1$ in $\RG_B(\cC)$, described by the following diagram:
\[
\xymatrix{
K_1 \ar@<1ex>[dr]^{p} \ar@<-1ex>[dr]_{p} \ar@{ >->}[rr]^{k_1} & & X_1 \ar@<1ex>[dl]^{c} \ar@<-1ex>[dl]_{d} \\
& B \ar[ul]|s \ar[ur]|e
}
\]
where $s$ is the unique arrow such that $k_1s=e$ and $ps=1_B$.

Taking the kernels in \cC\ of the domain projections of $K_1$, $X_1$ and $Y_1$, and the morphisms between them induced by $f_1$ and $k_1$, we get the pullback squares
\[
\xymatrix{
K \ar@{ |>->}[r]^-{k} \ar@{ |>->}[d]_{j=\ker(p)} & X \ar@{->>}[r]^{f} \ar@{ |>->}[d]_{h=\ker(d)} & Y \ar@{ |>->}[d]^{h'=\ker(d')} \\
K_1 \ar@{ >->}[r]_{k_1} & X_1 \ar@{->>}[r]_{f_1} & Y_1.
}
\]
It is easy to check that $k=\ker(f)$ and $hk=\ker(f_1)$, so that $K= K[f]=K[f_1]$ is indeed a normal subobject of $X_1$. The corresponding morphisms of precrossed modules will then look like
\[
\xymatrix{
K \ar[dr]_{0} \ar@{ |>->}[r]^-{k} & X \ar[d]_{\partial} \ar@{->>}[r]^{f} & Y \ar[dl]^{\partial'} \\
& B.
}
\]
We denote by $\psi\colon B\flat K\to K$ the action of $B$ on $K$ corresponding to the point $(p,s)$, which gives the precrossed module structure on $0\colon K \to B$.

It follows from Proposition \ref{farf} that the Peiffer commutator $\langle K,X \rangle$ is trivial if and only if there exists an arrow $\varphi$ making the diagram
\[
\xymatrix{
    K \ar[r]^-{l_{K}} \ar@{ |>->}[dr]_{k} & K \farf X \ar@{-->}[d]^{\varphi} & X \ar[l]_-{l_X} \ar[dl]^{1_X} \\
    & X
}
\]
commute. By precomposition, this in turn yields the (unique) dashed morphisms making the diagrams
\[
\xymatrix{
    K \ar[r]^-{(1,0)} \ar@{ |>->}[dr]_{k} & K \times X \ar@{-->}[d] & X \ar[l]_-{(0,1)} \ar[dl]^{1_X} \\
    & X
} \qquad
\xymatrix{
K \ar[r]^-{j_K} \ar@{ |>->}[dr]_{k} & K \rtimes_{\partial^*\psi} X \ar@{-->}[d] & X \ar[l]_-{i_X} \ar[dl]^{1_X} \\
& X
}
\]
commute, where in the left hand diagram we used the isomorphism $K \ltimes_{0^*\xi} X \cong K \times X$. So, in fact, the first diagram tells us that $K$ and $X$ commute in the sense of Huq, i.e.\ $[K,X]=[K[f_1],K[d]]=0$. On the other hand, the right hand diagram commutes if and only if the square
\[
\xymatrix@C=5ex{
X\flat K \ar[r]^{\ker[1,0]} \ar[d]_{\partial^*\psi} & X+K \ar[d]^{[1,k]} \\
K \ar[r]_{k} & X
}
\]
commutes (see \cite{J03}).
If we replace $\partial^*\psi$ by the conjugation action $\chi^X_K$ of $X$ on its normal subobject $K$, we get an analogous commutative diagram. As a consequence $\partial^*\psi=\chi^X_K$, since $k$ is a monomorphism.

Consider now the diagram
\begin{equation} \label{diag:two.actions}
\xymatrix@=7ex{
X\flat K \ar[r]^-{h\flat 1} \ar[dr]_{\partial^*\psi} & X_1\flat K \ar@<.5ex>[d]^(.4){\chi^{X_1}_K} \ar@<-.5ex>[d]_(.4){c^*\psi} & B\flat K \ar[l]_-{e\flat 1} \ar[dl]^{\psi} \\
& K,
}
\end{equation}
where $\chi^{X_1}_K$ denotes the conjugation action of $X_1$ on its normal subobject $K$. We want to show that both possible choices of the middle vertical arrow make the two triangles commute.

Let us start with the triangles on the left. The equality $\partial^*\psi=c^*\psi\cdot(h\flat 1)$ easily follows from the fact that $\partial=c\cdot h$, while the equality $\partial^*\psi=\chi^{X_1}_K\cdot(h\flat 1)$ holds because $\partial^*\psi=\chi^{X}_K$, as we proved above, and the commutative diagram
\[
\xymatrix@R=4ex{
& X_1\flat K \ar[rr]^-{\ker[1,0]} \ar[dd]_(.7){\chi^{X_1}_K} & & X_1+K \ar[dd]^{[1,hk]} \\
X\flat K \ar[rr]^(.65){\ker[1,0]} \ar[dd]_{\chi^{X}_{K}} \ar[ur]^{h\flat 1} & & X+K \ar[ur]^-{h+1} \ar[dd]^(.3){[1,k]} \\
& K \ar[rr]_(.3){hk} & & X_1 \\
K \ar[rr]_-{k} \ar@{=}[ur] & & X \ar[ur]_{h}
}
\]
shows that $\chi^{X}_{K}=\chi^{X_1}_K\cdot(h\flat 1)=h^*\chi^{X_1}_K$, since by composing with the monomorphism $kh$ they are equal.

As for the right hand triangles, by definition of pullback action we have $c^*\psi=\psi\cdot(c\flat 1)$, hence $c^*\psi\cdot(e\flat 1)=\psi\cdot(c\flat 1)\cdot(e\flat 1)=\psi$. On the other hand, the diagram
\[
\xymatrix@R=4ex{
& X_1\flat K \ar[rr]^-{\ker[1,0]} \ar[dd]_(.7){\chi^{X_1}_K} & & X_1+K \ar[dd]^{[1,k_1j]} \\
B\flat K \ar[rr]^(.65){\ker[1,0]} \ar[dd]_{\psi} \ar[ur]^{e\flat 1} & & B+K \ar[ur]^-{e+1} \ar[dd]^(.3){[s,j]} \\
& K \ar[rr]_(.3){k_1 j} & & X_1 \\
K \ar[rr]_-{j} \ar@{=}[ur] & & K_1 \ar[ur]_{k_1}
}
\]
shows that $\psi=\chi^{X_1}_K\cdot(e\flat 1)$, since by composing with the monomorphism $k_1j$ they are equal.

By \UA, since the cospan $(h,e)$ is extremal epimorphic by protomodularity (see Lemma 3.1.22 in \cite{BB}), the above arguments imply that $c^*\psi=\chi^{X_1}_K$.

Finally, if we consider $K$ as a (normal) subobject of $K[c]$:
\[
\xymatrix{
K \ar@{ |>->}[r] \ar@{ |>->}[d]_{j=\ker(p)} & K[c] \ar@{ |>->}[d]^{l=\ker(c)} \\
K_1 \ar@{ >->}[r]_{k_1} & X_1,
} 
\]
we get, as before, that $\chi^{K[c]}_K=l^*\chi^{X_1}_K$. Hence $\chi^{K[c]}_K=l^*\chi^{X_1}_K=l^*c^*\psi=0^*\psi$, which means that the conjugation action of $K[c]$ on $K$ is trivial, i.e.\ $[K,K[c]]=[K[f_1],K[c]]=0$.

Thanks to the characterization (\ref{char.centr}), this proves that any extension $f$ of precrossed $B$-modules as in \eqref{diag:f} is central with respect to \eqref{adjunction-PX} if the Peiffer commutator $\langle K[f],X\rangle$ is trivial.
\end{proof}

The previous characterization of central extensions, together with the properties of the Peiffer commutator, yields the following result.

\begin{coro}\label{coro:centralization}
If $f$ is an extension in $\PXB$ as in \eqref{diag:f}, then the induced extension
\[
\xymatrix{\frac{X}{\langle K[f],X\rangle} \ar@{->>}[rr]^-{\overline{f}} \ar[dr]_{\overline{\partial}} &  & Y  \ar[dl]^{\partial'} \\
&B & }
\]
is central and, moreover, any morphism $h$ from $f$ to a central extension $g$ factors uniquely through $\overline{f}$.
Accordingly, the category of $\XMod_B(\C)$-central extensions in $\PXMod_B(\mathcal C)$ is a reflective subcategory of the category of extensions in $\PXMod_B(\mathcal C)$.
\end{coro}

\begin{proof}
First observe that the extension $\overline{f}$ is central. Indeed, if we write $\eta \colon X \rightarrow \frac{X}{\langle K[f],X\rangle}$ for the canonical quotient, then 
$$ \langle K[\overline{f}] , \frac{X}{\langle K[f],X\rangle} \rangle = \langle \eta( K[f]) , \eta (X)\rangle  
 =  \eta \langle K[f], X \rangle 
 = 0,$$
where we have used the property of preservation of the Peiffer commutator by regular images (\ref{prop.peiffer}). Let then $h \colon X \rightarrow Z$ be a morphism in $\PXB$ from $f$ to another central extension $g \colon Z \rightarrow Y$, so that $g h = f$. Consider the factorization of $h$ in $\mathcal C$ as a regular epimorphism $q$ followed by a monomorphism $i$:
$$
\xymatrix{X \ar[rr]^h \ar@{->>}[rd]_q  & & Z \\ 
& {I\, \, } \ar@{>->}[ur]_i & }
$$
To show that $h$ factors through $\eta$ it suffices to prove that $q$ factors through $\eta$. First observe that the induced morphism $g i \colon I \rightarrow Y$ is a central extension, i.e. $\langle K[g i], I \rangle = 0$ (this follows immediately from Proposition $3.13$ in \cite{CMM17}).
By applying once again the property of preservation of the Peiffer commutator by regular images this implies that $$q(\langle K[f], X \rangle) = \langle q(K[f]), q(X) \rangle =  \langle K[gi] , I \rangle = 0.$$
The last statement is then clear, since we have just proved that $\eta$ satisfies the universal property of the $f$-component of the unit of the reflection into the subcategory of $\XMod_B(\C)$-central extensions in $\PXMod_B(\mathcal C)$.
\end{proof}

Since quotienting by $\langle X, K[f]\rangle$ gives the centralization of an extension $f$, under the hypotheses of Theorem \ref{thm:main}, the normal sub-precrossed module $(0\colon \langle X, K[f]\rangle\to B,\xi_{\langle X, K[f]\rangle})$ coincides with the relative commutator $[X,K[f]]_{\PXB}$ as defined in \cite{EG07}.

\section{Hopf formula for the fundamental group and homology}

Given a normal extension $f\colon X\to Y$ in $\PXB$, its \emph{Galois groupoid} is defined (see for example \cite{J08,JK}) as the reflection of its kernel pair $(Eq[f],p_1,p_2)$ into $\XB$. By analogy with the pointed case, we call the intersection of the kernels of $G(p_1)$ and $G(p_2)$ the \emph{Galois group} of $f$ and denote it by $\mathrm{Gal}(f,0)$. This is equivalent to the kernel of the normalization of the Galois groupoid, i.e.~of the composite $G(p_2) \ker (G(p_1)) \colon K[G(p_1)] \rightarrow G(X)$. Since $f$ is a normal extension, the square
\[\xymatrix{ Eq[f] \ar@{->>}[d]_{\eta_{Eq[f]}} \ar[r]^-{p_1} & X \ar@{->>}[d]^{\eta_X} \\
G(Eq[f]) \ar[r]_-{G(p_1)} & G(X)
}\]
is a pullback, and thus $\ker(G(p_1))$ is equal to $\eta_{Eq[f]}\ker(p_1)$. We then have
\[G(p_2)\ker(G(p_1)) = G(p_2) \eta_{Eq[f]}\ker(p_1)=\eta_X p_2 \ker(p_1)=\eta_X \ker(f),\]
and thus
\[\mathrm{Gal}(f,0)=K[\eta_X \ker(f)]=K[f]\wedge K[\eta_X]=K[f]\wedge \langle X, X\rangle.\]

Let us assume that the category $\PXB$ has enough (regular) projectives. This is the case, for instance, whenever $\mathcal C$ is a semi-abelian variety (see for example \cite{EG06}). For a given precrossed module $(\partial\colon X\to B, \xi_X)$, we can then consider a regular epimorphism
\[
p\colon (\partial_P\colon P\to B, \xi_P)\to (\partial\colon X\to B, \xi_X) 
\]
with $(\partial_P\colon P\to B, \xi_P)$ a projective precrossed module, and then its centralisation
\[
\xymatrix{
(\partial_P\colon P\to B, \xi_P) \ar@{->>}[dr]_p \ar[rr]^q & & (\overline{\partial_{P}}\colon \frac{P}{\langle P,K[p]\rangle}\to B, \overline{\xi_{P}}) \ar@{->>}[dl]^{\overline{p}} \\
& (\partial\colon X\to B, \xi_X)
}
\]
in $\PXB$. Since $(\partial_P\colon P\to B, \xi_P)$ is projective, thanks to the universal property of the centralization expressed by Corollary \ref{coro:centralization}, one can show that $\overline{p}$ is a \emph{weakly universal central extension}: for any other central extension $c\colon(\partial'\colon Y\to B,\xi_Y)\to(\partial\colon X\to B, \xi_X)$, there exists a morphism of precrossed modules $t\colon(\overline{\partial_{P}}\colon \frac{P}{\langle P,K[p]\rangle}\to B, \overline{\xi_{P}})\to(\partial'\colon Y\to B,\xi_Y)$ such that $ct=\overline{p}$. In our context, such a universal central extension is in fact normal, so that we can consider its fundamental groupoid. The Galois groupoid of $(\partial\colon X\to B, \xi_X)$ can then be defined as the Galois groupoid of $\overline{p}$ since, according to \cite{J08}, it does not depend on the choice of the weakly universal normal extension of $(\partial\colon X\to B, \xi_X)$. The fundamental group $\pi_1(\partial\colon X\to B, \xi_X)$ is the Galois group $\mathrm{Gal}(\overline{p},0)$. This is given as above by the formula 
\[
\pi_1(\partial_X\colon X\to B,\xi_X)=\mathrm{Gal}(\overline{p},0)=K[\overline{p}] \wedge \left\langle \frac{P}{\langle P,K[p]\rangle},\frac{P}{\langle P,K[p]\rangle}\right\rangle.
\]
Since the Peiffer commutator is preserved by regular images, we have
\[
\left\langle \frac{P}{\langle P, K[p]\rangle},\frac{P}{\langle P,K[p]\rangle} \right\rangle =\frac{\langle P,P\rangle}{\langle P,K[p]\rangle}.
\]
Moreover, since we have a regular epimorphism
\[
\overline{p}\colon\frac{P}{\langle P,K[p]\rangle} \to X\cong \frac{P}{K[p]},
\]
the Noether isomorphism theorem (see Theorem 2.2 in \cite{EG07}) gives us
\[
K[\overline{p}] = \frac{K[p]}{\langle P,K[p]\rangle}.
\]
To sum up, we find that the Galois group of the precrossed module $(\partial\colon X\to B,\xi)$ is given by the Hopf formula
\[
\pi_1(\partial_X\colon X\to B,\xi_X)\cong \frac{K[p]\wedge \langle P,P\rangle}{\langle P,K[p]\rangle},
\]
which is also the second homology object ${H}_2 ( X, \partial)$ of $(\partial\colon X\to B,\xi)$ as defined in \cite{EG07}.

Recall that two composable arrows in $\PXB$
\[
\xymatrix{(K, \partial_K) \ar[r]^f & (X, \partial_X) \ar[r]^g & (Y, \partial_Y) }
\]
form a \emph{short exact sequence} in $\PXB$ if $f = \ker (g)$ and $g$ is a regular epimorphism. Notice that, in this case, $\partial_K = 0$. A diagram 
\[
\xymatrix{(X_{i-1}, \partial_{i-1}) \ar[r]^-{f_{i-1}} & (X_i, \partial_i) \ar[r]^-{f_{i}} & (X_{i+1}, \partial_{i+1}) }
\]
is an \emph{exact sequence} if 
\[
\xymatrix{(I(f_{i-1}), \partial_{i} m_{i-1}) \ar[r]^-{m_{i - 1}} & (X_i, \partial_i) \ar[r]^-{p_{i}} & (I(f_{i}), \partial_{i+1}m_i) }
\]
is a short exact sequence, where $ \xymatrix{ X_j \ar[r]^-{p_j} & I_j \ar[r]^-{m_j} & X_{j+1}  }$ is the regular epi-mono factorization in $\mathcal C$ of the morphism $f_j$ \cite{CMM17}.
Given a short exact sequence 
\begin{equation} \label{ses}
\xymatrix{
0 \ar[r] & (K, \partial_K) \ar[r]^f & (X, \partial_X) \ar[r]^g & (Y, \partial_Y) \ar[r] & 0
}
\end{equation}
and a projective presentation $p\colon (P, \partial_P)\to (X, \partial_X)$ of $(X,\partial_X)$, this also gives a projective presentation $gp \colon  (P, \partial_P) \to (Y, \partial_Y)$ of  $(Y,\partial_Y)$. It follows that 
\[H_2(Y, \partial_Y) \cong \frac{K[gp]\wedge \langle P,P\rangle}{\langle P,K[gp]\rangle},\] and we then get the following extension of the Stallings-Stammbach theorem for precrossed modules (of groups) given in \cite{CE}:

\begin{prop}\label{SS}
Any short exact sequence \eqref{ses} in $\PXB$, with $p\colon (P, \partial_P)\to (X, \partial_X)$ a projective presentation of $(X,\partial_X)$, induces a five-term exact sequence 
\begin{equation}\label{SS-PXB}
\xymatrix{ \frac{K[p]\wedge \langle P,P\rangle}{\langle P,K[p]\rangle} \ar[r] &  \frac{K[gp]\wedge \langle P,P\rangle}{\langle P,K[gp]\rangle} \ar[r] & \frac{K}{\langle K,X \rangle} \ar[r] & \frac{X}{\langle X, X \rangle} \ar@{->>}[r]^{\overline{g}}  & \frac{Y}{\langle Y, Y \rangle} }
\end{equation}
where the morphism $\overline{g}$ is a regular epimorphism.
\end{prop}

\begin{proof}
This follows from Theorem $4.6$ in \cite{EG07} and the remarks above.
\end{proof}

\begin{rem}
\emph{Observe that, when $B=0$, all Peiffer commutators above coincide with Huq commutators (see \cite{CMM17}) and we recover from the above result the internal version of the classical Stalling-Stammbach theorem: a short exact sequence
\[
\xymatrix{ 0 \ar[r] & K \ar[r]^f & X \ar[r]^g & Y \ar[r] &  0}
\]
in a semi-abelian category \cC\ yields a five-term exact sequence 
\[
\xymatrix{ \frac{K[p]\wedge [ P,P] }{[ P,K[p]]} \ar[r] &  \frac{K[gp]\wedge [ P,P] }{[ P,K[gp]]} \ar[r] & \frac{K}{[ K,X ]} \ar[r] & \frac{X}{[X, X]} \ar[r] & \frac{Y}{[Y, Y ]} \ar[r] & 0 }
\]
(see also \cite{EG07}).}
\end{rem}

\begin{example}\label{Lie}
When $\C$ is the category $\Lie_K$ of Lie algebras over a field $K$, the classical notion of action coincides with the semi-abelian one. Accordingly, a \emph{Lie algebra precrossed module} is given by two Lie algebra homomorphisms $\partial \colon  X\to B$ and $\xi \colon  B\to \mathrm{Der}(X)$, where $\mathrm{Der}(X)$ is the Lie algebra of derivations of $X$, such that $\partial (\xi(b) (x))=[b,\partial (x)]$ for all $x\in X$ and $b\in B$. A \emph{Lie algebra crossed module} \cite{LR} is then a precrossed module where the Peiffer identity
\[
\xi(\partial(x))(y)=[x,y]
\]
holds for all $x,y\in X$. In this case the Peiffer commutator $\langle M,N\rangle$ of two precrossed $B$-submodules of $X$ is the Lie ideal of $X$ generated by the Peiffer elements
\[
[m,n]-\xi(\partial(m))(n)\quad\text{ and }\quad [n,m]-\xi(\partial(n))(m)
\]
where $m\in M$, $n\in N$. In particular, for a morphism
\[
f\colon (\partial\colon X\to B, \xi) \to (\partial'\colon Y\to B,\xi') 
\]
in $\mathbf{PXMod}_{B} (\Lie_K),$ we have $\partial(k)=\partial'f(k)=0$ for all $k\in K[f]$, so that the Peiffer commutator $\langle K[f],X\rangle$ is generated by the terms $[k,x]$ and $\xi(\partial(x))(k)$. It is thus the same ideal as in Example 5 of \cite{DG}, and thus we find the characterization of central extensions given there as a special case of Theorem \ref{thm:main}. Moreover, given a short exact sequence \eqref{ses} in the category of Lie algebra precrossed modules,
we obtain an exact sequence of Lie algebra precrossed modules
\[
H_2(X,\partial_X)\to H_2(Y,\partial_Y) \to \frac{K}{\langle K,X\rangle}\to \frac{X}{\langle X,X\rangle } \to \frac{Y}{\langle Y,Y\rangle} \to 0.
\]
\end{example}

\section{Double central extensions and homology}

Let us denote $\Ext(\PXB)$ the full subcategory of the arrow category of $\PXB$ whose objects are the regular epimorphisms, and $\CExt(\PXB)$ the full subcategory of $\Ext(\PXB)$ whose objects are the central extensions described in Theorem \ref{thm:main}. Then Corollary \ref{coro:centralization} shows that the subcategory $\CExt(\PXB)$ is reflective in $\Ext(\PXB)$, and we write
\[
G^1\colon \Ext(\PXB)\to \CExt(\PXB) 
\]
for the corresponding reflector.

Let us also recall that in any exact Mal'tsev category $\mathcal{A}$, a square of regular epimorphisms
\begin{equation} \label{eq:double_extension}
\xymatrix{
X \ar[r]^g \ar[d]_f & Z\ar[d]^{h} \\ Y \ar[r]_{j}& W
}
\end{equation}
is a pushout if and only if the induced map $X \rightarrow Y \times_W Z$ to the pullback of $h$ and $j$ is also a regular epimorphism (see Theorem 5.7 in \cite{CKP}); a commutative square with this property is often called a regular pushout or a double extension. The latter name is due to the fact that a square \eqref{eq:double_extension} in $\mathcal{A}$ can be seen as an arrow $(g,j)\colon f\to h$ in $\Ext(\mathcal{A})$, that plays the role of an extension between extensions. If we denote by $\E^1$ the class of double extensions, then the property recalled above allows us to prove that, much like regular epimorphisms in $\mathcal{A}$, double extensions are stable under pullback and closed under composition in $\Ext(\mathcal{A})$, and of course every isomorphism of $\Ext(\mathcal{A})$ is a double extension. Together with the subcategory $\CExt(\mathcal{A})$ of central extensions, which is always reflective when $\mathcal{X}$ is a Birkhoff subcategory of $\mathcal A$ as in \eqref{adjunction-JK}, this defines a Galois structure $\Gamma^1$ on $\Ext(\mathcal{A})$. The category $\Ext(\mathcal{A})$ is regular Mal'tsev, but not exact in general; nevertheless, it is still true that the Galois structure $\Gamma^1$ is admissible, and that every double extension is an effective descent morphism (see \cite{E14}). Thus we can again call \emph{trivial} a double extension $(g,j)\colon f\to h$ such that the naturality square
\[\xymatrix{f \ar[r]^{(g,j)} \ar[d]_{\eta^1_f} & h \ar[d]^{\eta^1_h} \\ G^1(f) \ar[r]_{G^1(g,j)}& G^1(h)}\]
is a pullback in $\Ext(\mathcal{A})$. When $\mathcal{A}=\PXB$ and $\mathcal{X}=\XB$, this is equivalent to the square
\[
\xymatrix{
X \ar[r]^{g} \ar[d]_{} & Z \ar[d]^{} \\ \frac{X}{\langle X, K[f]\rangle } \ar[r]_{\overline{g}}& \frac{Z}{\langle Z, K[h]\rangle}
}
\]
being a pullback, where the vertical arrows are the canonical quotients. Then a \emph{double central extension} is a double extension that is ``locally trivial'', i.e.~such that there exists a double extension $(p,q)\colon r\to h$ for which the pullback of $(g,j)$ along $(p,q)$, which is the back face of the cube
\[
\xymatrix{
X \times_Z U \ar[dr]_{p_1} \ar[dd]_{f\times_h r} \ar[rr]_{p_2} & & U \ar[dr]^{p} \ar[dd]^(.6){r} & \\
& X \ar[dd]_(.3){f} \ar[rr]^(.3){g} & & Z \ar[dd]^{h} \\
Y \times_W V \ar[rr]_{p_2} \ar[dr]_{p_1} & & V \ar[dr]_{q} & \\
& Y \ar[rr]_{j} & & W,
}
\]
is a double trivial extension.

Notice that a double extension \eqref{eq:double_extension} can also be seen as an extension $(f,h)\colon g\to j$ ; it turns out that centrality is independent of the orientation, since a double extension is central as an extension $g\to j$ if and only if it is central as an extension $f\to h$ (although this is not true for triviality of double extensions) \cite{E10}.

In \cite{DG}, a characterization of double central extensions for the adjunction \eqref{adjunction-RG} using Smith-Pedicchio commutators was given. This allows to state the corresponding result for the adjunction \eqref{adjunction-PX}: 

\begin{theorem} \label{double-central}  Let $\mathcal C$ be a semi-abelian category satisfying $\rm(UA)$, and let 
\begin{equation} \label{eq:double_ext_PX}
\xymatrix{
(\partial_X\colon X\to B,\xi_X) \ar[d]_f \ar[r]^g & (\partial_Z\colon Z\to B,\xi_Z) \ar[d]^h\\ (\partial_Y\colon Y\to B,\xi_Y) \ar[r]_j & (\partial_W\colon W\to B,\xi_W)
}
\end{equation}
be a double extension in the category $\PXB$. Then \eqref{eq:double_ext_PX} is a double central extension if and only if
\[
\langle K[f]\wedge K[g],X\rangle = 0 = \langle K[f],K[g]\rangle.
\]
\end{theorem}

\begin{proof}
By Corollary 3 of \cite{DG}, the double extension
\[
\xymatrix{
X_1 \ar[r]^{g_1} \ar[d]_{f_1} & Z_1 \ar[d]^{h_1} \\
Y_1 \ar[r]_{j_1} & W_1
}
\]
of reflexive graphs corresponding to \eqref{eq:double_ext_PX} is central if and only if it satisfies the conditions
\[
[Eq[f_1]\wedge Eq[g_1],Eq[c]\vee Eq[d]]_{Smith} = \Delta_{X_1}=[Eq[f_1],Eq[g_1]]_{Smith}.
\]
The equality on the left may be interpreted as requiring that the comparison map $\langle f_1,g_1\rangle \colon X_1\to Y_1 \times_{W_1}Z_1$ is, according to the characterization \eqref{centralRG}, a central extension. By equivalence, the corresponding morphism $\langle f,g\rangle \colon X\to Y \times_{W}Z$ in $\PXB$ is a central extension, which means, by Theorem \ref{thm:main}, that
\[
\langle K[f]\wedge K[g],X\rangle = 0.
\]
Under the $\SH$ condition, the equality on the right is equivalent to the Huq commutator of $K[f]$ and $K[g]$ being trivial in $X_1$. But since $K[f]$ and $K[g]$ are subobjects of $X=K[d]$, this is is equivalent to their Huq commutator being trivial in $X$. This in turn implies that
\[
\langle K[f],K[g] \rangle = 0
\]
by Remark \ref{rem:normal}, since $K[f]$ and $K[g]$ are normal precrossed submodules of $X$.
\end{proof}

As for the characterization of central extensions, by the previous result, we get a description of the reflection of double extensions into the subcategory of double central extensions.

\begin{prop} \label{prop:double.centralization}
The centralization of a double extension as \eqref{eq:double_ext_PX} in $\PXB$ is given by
\begin{equation} \label{diag:double.central}
\xymatrix{
(\overline{\partial}\colon \frac{X}{\langle K[f]\wedge K[g],X\rangle \vee \langle K[f],K[g]\rangle} \to B,\overline{\xi}) \ar[d]_{\overline f} \ar[r]^-{\overline g} & (\partial_Z\colon Z\to B,\xi_Z) \ar[d]^h\\ (\partial_Y\colon Y\to B,\xi_Y) \ar[r]_j & (\partial_W\colon W\to B,\xi_W).
}
\end{equation}
\end{prop}

\begin{proof}
Let us first observe that, by Lemma \ref{lemma:normal} and Proposition \ref{prop.peiffer}:
\[
\langle K[f]\wedge K[g],X\rangle \vee \langle K[f],K[g]\rangle \leq K[f]\wedge K[g].
\]
Hence, denoting $J=\langle K[f]\wedge K[g],X\rangle \vee \langle K[f],K[g]\rangle$ and $q\colon X \to X/J$, we have a pushout
\[
\xymatrix{
X \ar@{->>}[d]_{q} \ar@{->>}[r] & \frac{X}{K[f]\wedge K[g]} \ar@{=}[d] \\
\frac{X}{J} \ar@{->>}[r] & \frac{X}{K[f]\wedge K[g]},
}
\]
and as a consequence, by taking kernels horizontally:
\[
q(K[f]\wedge K[g]) =\frac{K[f]\wedge K[g]}{J}=\frac{K[f]}{J}\wedge\frac{K[g]}{J}=q(K[f])\wedge q(K[g]).
\]
We are going to show that the double extension \eqref{diag:double.central} is central by means of the characterization given by Theorem \ref{double-central}:

\begin{align*}
\left\langle K[\overline{f}]\wedge K[\overline{g}],\frac{X}{J}\right\rangle & = \left\langle \frac{K[f]}{J}\wedge \frac{K[g]}{J},\frac{X}{J}\right\rangle = \\
& = \langle q(K[f])\wedge q(K[g]),q(X)\rangle = \\
& = \langle q(K[f] \wedge K[g]),q(X)\rangle = \\
& = q(\langle K[f] \wedge K[g],X\rangle) = 0,
\end{align*}

\begin{align*}
\langle K[\overline{f}], K[\overline{g}]\rangle & = \left\langle \frac{K[f]}{J}, \frac{K[g]}{J}\right\rangle = \\
& = \langle q(K[f]), q(K[g]) \rangle = q(\langle K[f], K[g] \rangle) = 0.
\end{align*}

By the results of \cite{E14}, we know that the category of double central extensions is reflective in the category of double extensions in $\PXMod(\C)/B$. Moreover, by a result due to Im and Kelly \cite{IK}, the reflection must fix all but the ``top object'' (here $X$) of the double extension. To prove the universal property, there is then no restriction in considering an arrow $\phi$ of the form
\[
\xymatrix{
X \ar[dr]^{\phi} \ar[dd]_f \ar[rr]^g & & Z \ar@{=}[dr] \ar[dd] \\
& X' \ar[dd]_(.3){f'} \ar[rr]^(.3){g'} & & Z \ar[dd] \\
Y \ar[rr] \ar@{=}[dr] & & W \ar@{=}[dr] \\
& Y \ar[rr] & & W
}
\]
where the front face is a double central extension. Consider the decomposition $\phi=ip$, where $i$ is a monomorphism and $p$ is a regular epimorphism. Then it induces a diagram
\[
\xymatrix{
X \ar@{->>}[dr]^{p} \ar[dd]_f \ar[rr]^g & & Z \ar@{=}[dr] \ar[dd] \\
& I \ar[dd]_(.3){f' i} \ar[rr]^(.3){g' i} & & Z \ar[dd] \\
Y \ar[rr] \ar@{=}[dr] & & W \ar@{=}[dr] \\
& Y \ar[rr] & & W
}
\]
where the front face is a double extension, since $p$ is a regular epimorphism and the back face is a double extension. Moreover, it is a double central extension since $i$ is a monomorphism and double central extensions are closed under subobjects in double extensions. Then
\begin{align*}
p(J) & = p(\langle K[f]\wedge K[g],X\rangle) \vee p(\langle K[f],K[g]\rangle) \\
& = \langle p(K[f]\wedge K[g]),p(X)\rangle \vee \langle p(K[f]),p(K[g])\rangle \\
& = \langle K[f' i] \wedge K[g' i],I \rangle \vee \langle K[f' i],K[g' i]\rangle \\
& = 0,
\end{align*}
where the first equality follows from the fact that regular images distribute over joins. So $p$ factors through $q$ yielding a commutative triangle of double extensions, which shows that $q$ gives indeed the required reflection.
\end{proof}

%\begin{coro}For a double extension \eqref{eq:double_ext}, let $M$ be the subobject $\langle Ker(f)\wedge Ker(g),X\rangle \vee [Ker(f),Ker(g)]^{X}_{Huq}$. Since $M\leq Ker(f)$ and $M\leq Ker(g)$, we have induced maps $\overline{f}\colon \frac{X}{M}\to Y$, $\overline{g}\colon \frac{X}{M}\to Z$ and $\overline{\partial_X}\colon \frac{X}{M}\to B$; then the square
%\[\xymatrix{(\overline{\partial_X}\colon  X/M \to B,\overline{\xi_X}) \ar[d]_{\overline{f}} \ar[r]^{\overline{g}} & Z \ar[d]^h\\ Y \ar[r]_j & W}\]
%is a double extension, and it is the reflection of \eqref{eq:double_ext} in the subcategory of double central extensions.
%\end{coro}

In particular, if we consider two normal precrossed submodules $(0\colon H \to B,\xi_H)$ and $(0\colon K \to B,\xi_K)$ of a given precrossed module $(\partial\colon X \to B,\xi)$, then the join $H\vee K$ in \cC\ is endowed with a precrossed module structure over $B$, and it is normal in $(\partial\colon X \to B,\xi)$ too (see \cite{CMM17}). One can then consider the double extension
\[
\xymatrix{
(0 \colon H \vee K \to B,\xi_{H\vee K}) \ar[d] \ar[r] & (0 \colon \frac{H\vee K}{H} \to B,\xi_{\frac{H\vee K}{H}}) \ar[d] \\
(0 \colon \frac{H\vee K}{K} \to B,\xi_{\frac{H\vee K}{K}}) \ar[r] & (0 \colon 0 \to B,\tau),
}
\]
and apply Proposition \ref{prop:double.centralization} to this special case, whose centralization is obtained by quotienting out the object
\[
\langle H \wedge K , H \vee K \rangle \vee \langle H , K \rangle = \langle H \wedge K , H \rangle \vee \langle H \wedge K , K \rangle \vee \langle H , K \rangle = \langle H , K \rangle.
\]
Let us observe that $H$ and $K$ act trivially on each other by the action induced by $B$, because their structure maps are zero, whence
\[
\langle H , K \rangle = [H, K]_{Huq}
\]
by Remark 3.12 in \cite{CMM17}.

Finally, slightly enlarging the context of \cite{EVdL} to include quasi-pointed categories, we may say that the centralization just described provides a description of the relative commutator of two normal precrossed submodules with respect to the adjunction \eqref{adjunction-PX}, so that

\[
[H,K]_{\PXB}=(0\colon [H,K]_{Huq}\to B,\xi_{[H,K]}).
\]

\section{{The third homology object}}

{Following the lines of Section 6 in \cite{J08} and} using the characterization of double central extensions we are now going to establish a Hopf formula for the third homology object in $\PXB$, which {specializes in particular to the third integral homology group of a group} \cite{BE}. To this purpose, we assume again that $\PXB$ has enough regular projectives, and we can first define $\pi_2(\partial\colon X\to B, \xi_X)$ as the Galois group of a weakly universal double central extension. To construct such a double extension, we take two projective precrossed modules $(\partial_P\colon P\to B, \xi_P)$ and $(\partial_{P'}\colon {P'}\to B, \xi_{P'})$ and regular epimorphisms $p\colon P\to X$ and $p'\colon P'\to X$; then we form the pullback $P\times_{X}P'$ of $p$ and $p'$, and take a projective precrossed module $(\partial_{Q}\colon {Q}\to B, \xi_{Q})$ with a regular epimorphism $Q\to P\times_{X}P'$. The square
\begin{equation} \label{double-ext}
\xymatrix{
Q \ar@{->>}[r]^{q} \ar@{->>}[d]_{q'} & P' \ar@{->>}[d]^{p'} \\ P \ar@{->>}[r]_{p} & X
}\end{equation}
is then a double extension (in $\PXB$), so that we can see $q\to p$ and $q'\to p'$ as extensions with projective domains in the category $\Ext(\PXB)$. As in the one-dimensional case, the centralisation
\[
\xymatrix{
\frac{Q}{\langle K[q]\wedge K[q'],Q\rangle\vee \langle K[q],K[q']\rangle} \ar@{}[r]|(.8){=} & \overline{Q} \ar@{->>}[r]^-{\overline{q}} \ar@{->>}[d]_{\overline{q'}} & P' \ar@{->>}[d]^{p'} \\
& P \ar@{->>}[r]_{p} & X
}
\]
of this double extension is then a weakly universal double central extension, and we can use it to compute the fundamental group of the extension $p'$ as

\begin{align*}
\pi_1(p')& = K[(\overline{q},p)]\wedge K[\eta^1_{\overline{q'}}] \\
& = K[\overline{q}] \wedge \langle K[\overline{q'}],\overline{Q} \rangle \to 0 \\
& = \frac{K[q]\wedge \langle K[q'],Q\rangle}{\langle K[q]\wedge K[q'],Q\rangle\vee \langle K[q],K[q']\rangle }\to 0
\end{align*}

where the second equality is explained by the following (horizontal) pullback in $\Ext(\PXB)$:
\[
\xymatrix{
& K[\overline{q}] \wedge \langle K[\overline{q'}],\overline{Q} \rangle \ar[dl] \ar[rr] \ar[dd] & & \langle K[\overline{q'}],\overline{Q} \rangle \ar[dd]^{K[\eta^1_{_{\overline{q'}}}]} \ar[dl]_(.7){\ker(\eta^1_{_{\overline{q'}}})} \\
K[\overline{q}] \ar[dd]_{K[(\overline{q},p)]} \ar[rr] & & \overline{Q} \ar[dd]^(.3){\overline{q'}} \\
& 0 \ar[dl] \ar@{=}[rr] & & 0 \ar[dl] \\
K[p] \ar[rr] & & P.
}
\]

Analogously
\[
\pi_1(p)=\frac{K[q']\wedge \langle K[q],Q\rangle}{\langle K[q]\wedge K[q'],Q\rangle\vee \langle K[q],K[q']\rangle }\to 0.
\]
Since $q\colon Q\to P'$ is also a regular epimorphism with projective domain, the fundamental group of $P'$ can be calculated as
\[
\pi_1(P')=\frac{K[q]\wedge \langle Q,Q\rangle}{\langle K[q],Q\rangle},
\]
but since $P'$ is projective, this fundamental group must be trivial, which implies that $K[q]\wedge \langle Q,Q \rangle=\langle K[q],Q\rangle$. By analogy, we must also have that $K[q']\wedge \langle Q,Q\rangle=\langle K[q'],Q\rangle$, and as a consequence, we obtain
\[
\pi_1(p)=\frac{K[q']\wedge K[q]\wedge \langle Q,Q\rangle}{\langle K[q]\wedge K[q'],Q\rangle\vee \langle K[q],K[q']\rangle }\to 0=\pi_1(p').
\]
Since this must be true for any $p$ and $p'$, and $\pi_1(p)$ and $\pi_1(p')$ only depend on $p$ and $p'$ respectively, $\pi_1(p)$ only depends on its codomain $(\partial \colon X\to B,\xi)$; thus we can define $\pi_2(\partial\colon X\to B,\xi)$ as the domain of $\pi_1(p)$, and this gives us the Hopf formula
\begin{equation} \label{double-Hopf}
H_3 (X, \partial ) = \pi_2(X, \partial )=\frac{K[q']\wedge K[q]\wedge \langle Q,Q\rangle}{\langle K[q]\wedge K[q'],Q\rangle\vee \langle K[q],K[q']\rangle},
\end{equation}
which is independent of the chosen double extension \eqref{double-ext}.
%These turn out to be isomorphic, so that we can define $\pi_2(X)$ as $\pi_1(p\colon P\to X)$, which is then given by another Hopf formula.


\begin{thebibliography}{99}

\bibitem{Barr} \textsc{M.~Barr}, Exact categories, \textit{Lecture Notes in Mathematics} \textbf{236} (1971), 1--120.

\bibitem{BB} \textsc{F.~Borceux and D.~Bourn}, \textit{Mal'cev, protomodular, homological and semi-abelian categories}, Math. Appl., vol. 566, Kluwer Acad. Publ., 2004.

\bibitem{BJK}  \textsc{F.~Borceux, G.~Janelidze, and G.~M.~Kelly}, Internal object actions, \textit{Comment. Math. Univ. Carolinae} \textbf{46} no. 2 (2005), 235--255.

\bibitem{Bourn1991} \textsc{D.~Bourn}, Normalization equivalence, kernel equivalence and affine categories, \textit{Lecture Notes in Mathematics} \textbf{1488} (1991), 43--62.

\bibitem{Bourn2001}  \textsc{D.~Bourn}, $3 \times 3$ lemma and protomodularity, \textit{J. Algebra} \textbf{236} (2001), 778--795.

\bibitem{Bourn2002} \textsc{D.~Bourn}, Intrinsic centrality and associated classifying properties, \textit{J. Algebra} \textbf{256} (2002) 126-145.

\bibitem{Bourn2004} \textsc{D.~Bourn}, Commutator theory in regular Mal'cev categories, in Hopf Algebras and Semi- abelian Categories. G. Janelidze, B. Pareigis, W. Tholen eds., \textit{the Fields Institute Commun.} \textbf{43}, Amer. Math. Soc., 61--75 (2004).

\bibitem{BG} \textsc{D.~Bourn and M.~Gran}, Centrality and normality in protomodular categories, \textit{Theory Appl. Categ.} \textbf{9}(8)  (2002), 151--165.

\bibitem{BG3} \textsc{D.~Bourn and M.~Gran}, Centrality and connectors in Mal'tsev categories, \textit{Algebra Universalis} \textbf{48} (2002), 309--331.

\bibitem{BJ} \textsc{D.~Bourn and G.~Janelidze}, Protomodularity, descent, and semidirect products, \textit{Theory Appl. Categ.} \textbf{4} (1998), 37--46.


\bibitem{BE} \textsc{R. Brown, G.J. Ellis}, {Hopf formulae for the higher homology of a group},
   \textit{Bull. London Math. Soc.} \textbf{20} (2), (1988) {124--128.}
 
\bibitem{CLP} \textsc{A.~Carboni, J.~Lambek and M.~C.~Pedicchio}, Diagram chasing in Mal'cev categories, \textit{Journal of Pure and Applied Algebra} \textbf{69} (1990), 271--284.

\bibitem{CKP} \textsc{A.~Carboni, G.~M.~Kelly, and M.~C.~Pedicchio}, Some remarks on Maltsev and Goursat categories, \textit{Applied Categorical Structures} \textbf{1} (1993), 385--421.

\bibitem{Cig} \textsc{A.~S.~Cigoli}, Descent cospans for the fibration of points, preprint, 2018.

\bibitem{CMM17} \textsc{A.~S.~Cigoli, S.~Mantovani, G.~Metere}, Peiffer product and Peiffer commutator for internal pre-crossed modules, \textit{Homology Homotopy Appl.} \textbf{19} (2017) 181--207.

\bibitem{CE} \textsc{D.~Conduch\'e and G.~J.~Ellis}, Quelques propri\'et\'es homologiques des modules pr\'ecrois\'es, \textit{J.~Algebra} \textbf{123} (1989), 327--335.

\bibitem{DG} \textsc{A.~Duvieusart and M.~Gran}, Higher commutator conditions for extensions in Mal'tsev categories, \textit{J.~Algebra} \textbf{515} (2018), 298--327.

\bibitem{E10} \textsc{T.~Everaert}, Higher central extensions and Hopf formulae, \textit{J.~Algebra} \textbf{324} (2010) 1771-1789

\bibitem{E14} \textsc{T.~Everaert}, Higher central extensions in Mal'tsev categories, \textit{Appl.~Categ.~Structures} \textbf{22} (2014), 961--979.

\bibitem{EG06} \textsc{T.~Everaert and M.~Gran}, Precrossed modules and Galois theory, \textit{J.~Algebra} \textbf{297} (2006) 292--309.

\bibitem{EG07} \textsc{T.~Everaert and M.~Gran}, On low-dimensional homology in categories, \textit{Homology Homotopy Appl.} \textbf{9} (2007) 275--293.

\bibitem{EVdL} \textsc{T.~Everaert and T.~Van der Linden}, Relative commutator theory in semi-abelian categories, \textit{J. Pure Appl. Algebra} \textbf{216} (2012) 1791--1806.

\bibitem{Huq} \textsc{S.~A.~Huq}, Commutator, nilpotency and solvability in categories, \textit{Quart.~J.~Math.~Oxford Ser.} \textbf{19} (1968) 363--389.

\bibitem{IK} \textsc{G.~B.~Im and G.~M.~Kelly}, On classes of morphisms closed under limits, \textit{J.~Korean Math.~Soc.} \textbf{23} (1986), 1--18.

\bibitem{J03} \textsc{G.~Janelidze}, Internal crossed modules, \textit{Georgian Math.~J.} \textbf{10} (2003) 99--114.

\bibitem{J08} \textsc{G.~Janelidze}, Galois Groups, Abstract Commutators, and Hopf Formula, \textit{Appl.~Categ.~Structures} \textbf{16} (2008) 653--668.

\bibitem{JK} \textsc{G.~Janelidze and G.~M.~Kelly}, Galois theory and a general notion of central extension \textit{J.~Pure Appl.~Algebra} \textbf{2} (1994), 135--161.

\bibitem{JMT} \textsc{G.~Janelidze, L.~M\'arki, and W.~Tholen}, Semi-abelian categories, \textit{J.~Pure Appl.~Algebra} \textbf{168} (2002), 367--386.

\bibitem{JT} \textsc{G.~Janelidze and W.~Tholen}, Facets of descent, I, \textit{Appl.~Categ.~Structures} \textbf{2} (1994), 245--281.

\bibitem{LR} \textsc{R.~Lavendhomme and J.-R.~Roisin}, Cohomologie non-ab\'elienne de structures alg\'ebriques, \textit{J.~Algebra} \textbf{67} (1980) 385--414.

\bibitem{MM} \textsc{S.~Mantovani and G.~Metere}, Internal crossed modules and Peiffer condition, \textit{Theory Appl.~Categ.} \textbf{23} (6) (2010), 113--135. 

\bibitem{NT} \textsc{N.~Martins-Ferreira and T.~Van der Linden}, A note on the ``Smith is Huq'' condition, \textit{Appl.~Categ.~Structures} \textbf{20} (2) (2012), 175--187.

\bibitem{Orzech} \textsc{G.~Orzech}, Obstruction theory in algebraic categories I and II, \textit{J.~Pure Appl.~Algebra} \textbf{2} (1972), 287--314 and 315--340.

\bibitem{Ped} \textsc{M.~C.~Pedicchio}, A categorical approach to commutator theory, \textit{J.~Algebra} \textbf{177} (3) (1995), 647--657.

\bibitem{Smith} \textsc{J.~D.~H.~Smith}, \textit{Mal'cev varieties}, Lecture Notes in Math.~554, Springer-Verlag (1976).

\end{thebibliography}
\end{document}